\RequirePackage{etex}
\documentclass[pdftex,a4paper]{article}
\usepackage{amsmath,amssymb,amsthm}
\usepackage[top=3cm,bottom=3cm,left=2.5cm,right=2.5cm]{geometry}
\usepackage{microtype}
\usepackage{tikz}
\usetikzlibrary{arrows,cd}
\usepackage[hypertexnames=false]{hyperref}
\hypersetup{
	colorlinks=true,
	citecolor=red,
	linkcolor=blue,
	urlcolor=orange
	}
\usepackage{pdflscape}
\usepackage{subfigure}
\usepackage{cleveref}
\usepackage{autonum}
\usepackage{enumitem}
\setlist[enumerate,1]{label=(\arabic*),ref=(\arabic*)}
\setlist[enumerate,2]{label=(\alph*),ref=\theenumi. (\alph*)}
\crefname{equation}{equation.}{equations.}

\numberwithin{equation}{subsection}
\theoremstyle{definition}

\newtheorem{theorem}{Theorem}[section]
\crefname{theorem}{Theorem}{Theorems}

\newtheorem{maintheorem}[theorem]{Main Theorem}
\crefname{maintheorem}{Main Theorem}{Main Theorems}

\newtheorem{lemma}[theorem]{Lemma}
\crefname{lemma}{Lemma}{Lemmas}
\newtheorem{proposition}[theorem]{Proposition}
\crefname{proposition}{Proposition}{Propositions}
\newtheorem{corollary}[theorem]{Corollary}
\crefname{corollary}{Corollary}{corollaries}
\newtheorem{definition}[theorem]{Definition}
\crefname{definition}{Definition}{Definitions}
\newtheorem{definitionproposition}[theorem]{Definition-Proposition}
\crefname{definitionproposition}{Definition-Proposition}{Definition-Proposirions}
\newtheorem{example}[theorem]{Example}
\crefname{example}{Example}{Examples}
\newtheorem{remark}[theorem]{Remark}
\crefname{remark}{Remark}{Remarks}

\newcommand{\setmid}{\; \middle|\;}
\newcommand{\lmod}{\operatorname{\mathrm{\hspace{-2pt}-mod}}}

\newcommand{\Kernel}{\operatorname{\mathrm{Ker}}}

\newcommand{\Hom}{\operatorname{Hom}\nolimits}

\newcommand{\induc}{{\operatorname{Ind}\nolimits}}
\newcommand{\restr}{{\operatorname{Res}\nolimits}}
\newcommand{\add}{\operatorname{\mathrm{add}}}
\newcommand{\stautilt}{\operatorname{\mathrm{s\tau-tilt}}}
\newcommand{\inertiagp}{I}

\newcommand{\Id}{\mathrm{Id}}

\newcommand{\Addresses}{{
  \bigskip
  \footnotesize
  Ryotaro~KOSHIO\par\nopagebreak
  \textsc{Department of Mathematics, Tokyo University of Science}
	\par\nopagebreak
	1-3, Kagurazaka, Shinjuku-ku, Tokyo, 162-8601, Japan
  \par\nopagebreak
  E-mail: \href{mailto:1120702@ed.tus.ac.jp}{1120702@ed.tus.ac.jp}
}}
\title{On induced modules of inertial-invariant support \(\tau\)-tilting modules over blocks of finite groups\footnote{\emph{Mathematics Subject Classification} (2020). 20C20, 16G10.
}
\footnote{\emph{Keywords.} Support \(\tau\)-tilting modules, Induced modules, Blocks of finite groups.}}

\author{Ryotaro~KOSHIO}
\date{\today}
\begin{document}
\maketitle
\begin{abstract}
	In this article, we prove that induced modules of support \(\tau\)-tilting modules over blocks of finite groups satisfying inertial-invariant condition are also support \(\tau\)-tilting modules.
\end{abstract}
\section{Introduction and notation}\label{intro 2021-11-04 11:17:58}
Support \(\tau\)-tilting modules introduced in \cite{MR3187626} form an important class of modules.
These correspond bijectively to various representation theorical objects, such as two-term silting complexes, functorially finite torsion classes, left finite semibricks, two-term simple-minded collections and more (see \cite{MR3187626,MR4139031,MR3220536,MR3178243}).
Let \(k\) be an algebraically closed field of characteristic \(p>0\), \(\tilde{G}\) a finite group, \(G\) a normal subgroup of \(\tilde{G}\), \(B\) a block of \(kG\) and \(\tilde{B}\) a block of \(k\tilde{G}\) covering \(B\), that is, the block of \(k\tilde{G}\) satisfying that \(1_B 1_{\tilde{B}} \neq 0\), where \(1_B\) and \(1_{\tilde{B}}\) mean the respective unit elements of \(B\) and \(\tilde{B}\).
We denote the inertial group of the block \(B\) in \({\tilde G}\) by \(\inertiagp_{\tilde G}(B)\)
and the second group cohomology of the factor group \(\inertiagp_{\tilde{G}}(B)/G\) with coefficients in the unit group \(k^\times\) of the field \(k\) with trivial action by \(H^2(\inertiagp_{\tilde{G}}(B)/G,k^\times)\).
In \cite{KK2}, construction-methods of support \(\tau\)-tilting modules over \(\tilde{B}\) from the ones over \(B\) using the induction functor \(\induc_G^{\tilde{G}}\) were presented under the following conditions:
\begin{enumerate}
	\item Any left finite brick \(U\) in the category of \(B\)-module is \(\inertiagp_{\tilde G}(B)\)-invariant, that is, \(xU\cong U\) as \(B\)-modules for any \(x\in \inertiagp_{\tilde G}(B)\).
	\item \(H^2(\inertiagp_{\tilde{G}}(B)/G,k^\times)=1\).
	\item The group algebra \(k[\inertiagp_{\tilde{G}}(B)/G]\) is basic as a \(k\)-algebra.
\end{enumerate}
This paper presents the following results, which relaxes the assumptions above.
First, we state a construction-method of support \(\tau\)-tilting modules over \(k\tilde{G}\) from the ones over \(kG\).
\begin{maintheorem}[{see \cref{main group algebra 2022-06-09 01:38:53,main block 2022-07-13 08:57:19}}]
	Let \(\tilde{G}\) be a finite group, \(G\) a normal subgroup of \(\tilde{G}\), \(B\) a block of \(kG\), \(\tilde{B}\) a block of \(k\tilde{G}\) covering \(B\) and \(M\) a support \(\tau\)-tilting \(B\)-modules satisfying \(xM\cong M\) as \(B\)-modules for any \(x\in \inertiagp_{\tilde G}(B)\).
	Then the induced module \(\induc_G^{\tilde{G}}M\) is a support \(\tau\)-tilting \(k\tilde{G}\)-module.
	In particular, the module \(\tilde{B}\induc_G^{\tilde{G}}M\) is a support \(\tau\)-tilting \(\tilde{B}\)-module.
\end{maintheorem}
We will demonstrate that there is an relation between \(\inertiagp_{\tilde{G}}(B)\)-invariant support \(\tau\)-tilting \(B\)-modules and support \(\tau\)-tilting \(k\tilde{G}\)-modules.
Now we recall that the set of support \(\tau\)-tilting \(\stautilt \Lambda\) has a partially ordered set structure for any finite dimensional algebra (see \cref{order 2022-08-30 16:55:02}).
\begin{maintheorem}[{see \cref{iff induc 2022-08-28 16:22:55}}]
	Let \(\tilde{G}\) be a finite group, \(G\) a normal subgroup of \(\tilde{G}\), \(B\) a block of \(kG\), \(\tilde{B}\) a block of \(k\tilde{G}\) covering \(B\) and \(M\) a \(B\)-modules satisfying \(xM\cong M\) as \(B\)-modules for any \(x\in \inertiagp_{\tilde G}(B)\).
	Then \(M\) is a support \(\tau\)-tilting \(B\)-module if and only if \(\induc_G^{\tilde{G}}M\) is a support \(\tau\)-tilting \(k\tilde{G}\)-module.
	Moreover, for any two \(\inertiagp_{\tilde G}(B)\)-invariant support \(\tau\)-tilting \(B\)-modules \(M\) and \(M'\), \(M\geq M'\) in \(\stautilt B\) if and only if \(\induc_G^{\tilde{G}}M\geq \induc_G^{\tilde{G}}M'\) in \(\stautilt k\tilde{G}\).
\end{maintheorem}

Throughout this paper, we use the following notation and terminologies.
\(k\) means an algebraically closed field of characteristic \(p>0\).
Let \(\Lambda\) be a finite dimensional algebra over a field \(k\).
Modules mean finitely generated left modules.
We denote by \(\Lambda\lmod\) the module category of \(\Lambda\).
For a \(\Lambda\)-module \(U\), we denote by \(P(U)\) the projective cover of \(U\), by \(\Omega(U)\) the syzygy of \(U\), by \(\tau U\) the Auslander--Reiten translate of \(U\) and by \(\add U\) the full subcategory of \(\Lambda\lmod\) whose objects are isomorphic to direct summands of finite direct sums of \(U\).

This paper is organized as follows.
In \cref{section preliminaries 2022-07-13 09:11:11}, we introduce basic terminologies and some known results for \(\tau\)-tilting theory and modular representation theory of finite groups.
In \cref{section main 2022-07-13 09:12:20}, we give some lemmas and the main result, and present applications and examples.

\section{Preliminaries}\label{section preliminaries 2022-07-13 09:11:11}
In this section, \(\Lambda\) means a finite dimensional \(k\)-algebra.
\subsection{Support \texorpdfstring{\(\tau\)}{tau}-tilting modules}
We recall the definitions and basic properties of support \(\tau\)-tilting modules.
For a \(\Lambda\)-module \(M\), we denote by \(|M|\) the number of isomorphism classes of indecomposable direct summands of \(M\).
In particular, \(|\Lambda|:=|{}_\Lambda \Lambda |\) means the number of isomorphism classes of simple \(\Lambda\)-modules.
\begin{definition}[{\cite[Definition 0.1]{MR3187626}}]
	Let \(\Lambda\) be a finite dimensional \(k\)-algebra and \(M\) a \(\Lambda\)-module.
	\begin{enumerate}
		\item We say that \(M\) is \(\tau\)-rigid if \(\Hom_{\Lambda}(M,\tau M)=0\).
		\item We say that \(M\) is \(\tau\)-tilting if \(M\) is a \(\tau\)-rigid module and \(|M|=|\Lambda|.\)
		\item We say that \(M\) is support \(\tau\)-tilting if there exists an idempotent \(e\) of \(\Lambda\) such that \(M\) is a \(\tau\)-tilting \(\Lambda/\Lambda e\Lambda \)-module.
	\end{enumerate}
\end{definition}
For support \(\tau\)-tilting \(\Lambda\)-modules \(M\) and \(M'\), we write \(M\sim_{\add} M'\) if \(\add M=\add M'\).
Then the relation \(\sim_{\add}\) is an equivalence relation.
We denote \(\stautilt \Lambda\) the set of equivalence classes of all support \(\tau\)-tilting \(\Lambda\)-modules under the equivalence relation \(\sim_{\add}\).
\begin{definitionproposition}[{\cite[Theorem 2.7]{MR3187626}}]\label{order 2022-08-30 16:55:02}
	For \(M, M' \in \stautilt \Lambda\), we write \(M\geq M'\) if there exist a positive integer \(r\) and an epimorphism
	\begin{equation}
		\begin{tikzcd}
			M^{\oplus r}\ar[r,twoheadrightarrow,"\varphi"]& M'.
		\end{tikzcd}
	\end{equation}
	Then we get a partial order on \(\stautilt \Lambda\).
\end{definitionproposition}
We denote by \(\mathcal{H}(\stautilt \Lambda)\) the Hasse diagram for the partially ordered set \(\stautilt \Lambda\).

\begin{remark}[{\cite[Proposition 2.3 (a), (b)]{MR3848421}}]\label{tau number remark 2021-09-07 12:11:53}
	Since \(e=0\) is an idempotent of \(\Lambda\) and \(\Lambda/\Lambda e\Lambda =\Lambda\), any \(\tau\)-tilting module is a support \(\tau\)-tilting module.
	Moreover, for any \(\tau\)-rigid \(\Lambda\)-module \(M\), the following conditions are equivalent:
	\begin{enumerate}
		\item \(M\) is a support \(\tau\)-tilting module.
		\item There exist a projective \(\Lambda\)-module \(P\) satisfying \(\Hom_\Lambda(P,M)=0\) and \(|M|+|P|=|\Lambda|\).
	\end{enumerate}
\end{remark}
\begin{proposition}[{\cite[Corollary 2.13]{MR3187626}}]\label{tau tilting pair 2022-08-27 08:26:25}
	Let \(M\) be a \(\tau\)-rigid \(\Lambda\)-module and \(P\) a projective \(\Lambda\)-module satisfying that \(\Hom_\Lambda(P,M)=0\).
	Then the following conditions are equivalent:
	\begin{enumerate}
		\item \(|M|+|P|=|\Lambda|\), that is, \(M\) is a support \(\tau\)-tilting \(\Lambda\)-module (see \cref{tau number remark 2021-09-07 12:11:53}).
		\item If \(\Hom_\Lambda(M,\tau X)=0\), \(\Hom_\Lambda(X,\tau M)=0\) and \(\Hom_\Lambda(P,X)=0\), then \(X\in \add M\) for any \(\Lambda\)-module \(X\).
	\end{enumerate}
\end{proposition}

The following proposition plays an important role in the proof of our main result.
\begin{proposition}[{\cite[Proposition 2.14]{MR3428959}}]\label{lem app 2022-06-02 16:07:34}
	Let \(\Lambda\) be a finite dimensional \(k\)-algebra and \(M\) a \(\tau\)-rigid \(\Lambda\)-module.
	Then \(M\) is a support \(\tau\)-tilting \(\Lambda\)-module if and only if there exists an exact sequence
	\begin{equation}
		\begin{tikzcd}
			\Lambda\ar[r,"f"]&M'\ar[r,"f'"]&M''\ar[r]&0
		\end{tikzcd}
	\end{equation}
	in \(\Lambda\lmod\) with \(M', M''\in \add M\) and \(f\) a left \(\add M\)-approximation of \(\Lambda\).
\end{proposition}

\subsection{Modules over blocks of finite group}
Let \(G\) be a finite group and \(H\) a subgroup of \(G\).
We denote by \(\restr_{H}^G\) the restriction functor from \(kG\lmod\) to \(kH\lmod\) and \(\induc_H^G:={}_{kG}kG\otimes _{kH} \bullet\) the induction functor from \(kH\lmod\) to \(kG\lmod\).
The field \(k\) can always be regarded as a \(kG\)-module by defining \(gx=x\) for any \(g\in G\) and \(x \in k\).
This module is called the trivial module and is denoted by \(k_G\).
\begin{proposition}[{see \cite[Lemma 8.5, Lemma 8.6]{MR860771}}]\label{Theorem: Frobenius and projective}
	Let \(G\) be a finite group, \(K\) a subgroup of \(G\), \(H\) a subgroup of \(K\).
	Then the following hold:
	\begin{enumerate}
		\item \(\restr^G_H \cong \restr^
		      K_H\restr^G_K\).
		\item \(\induc^G_H \cong \induc^G_K\induc^K_H\).\label{transitive induc 2022-06-09 02:03:23}
		\item The functors \(\restr_H^G\) and \(\induc_H^G\) are left and right adjoint to each other.\label{frob adjoint 2022-06-02 15:44:10}
		\item The functors \(\restr_H^G\) and \(\induc_H^G\) send projective modules to projective modules.\label{proj to proj 2022-06-02 12:44:53}
	\end{enumerate}
\end{proposition}

Let \(G\) be a normal subgroup of a finite group \(\tilde{G}\) and \(U\) a \(kG\)-module. For \(\tilde{g} \in \tilde{G}\), we define a \(kG\)-module \(\tilde{g}U\) consisting of symbols \(\tilde{g}u\) as a set, where \(u\in U\) and its \(kG\)-module structure is given by \(\tilde{g}u+\tilde{g}u':=\tilde{g}(u+u')\), \(\lambda(\tilde{g}u):=\tilde{g}(\lambda u)\) and \(g(\tilde{g}u):=\tilde{g}(\tilde{g}^{-1}g\tilde{g}u) \) for any \(u, u'\in U\), \(\lambda\in k\) and \(g \in G\).
Let \(U\) be a \(kG\)-module.
If \(U\) is projective or indecomposable, then \(xU\) is also projective or indecomposable, respectively.

\begin{theorem}[Mackey's decomposition formula for normal subgroups]\label{Mackey's decomposition formula 2022-08-28 15:16:37}
	Let \(G\) be a normal subgroup of a finite group \(\tilde{G}\) and \(U\) a \(kG\)-module. Then we have an isomorphism
	\begin{equation}
		\restr^{\tilde{G}}_G\induc^{\tilde{G}}_G U\cong\bigoplus_{x \in [\tilde{G}/G]} x U,
	\end{equation}
	of \(kG\)-modules, where \([\tilde{G}/G]\) is a set of representatives of the factor group \(\tilde{G}/G\).
\end{theorem}

We recall the definition of blocks of group algebras. Let \(G\) be a finite group. The group algebra \(kG\) has a unique decomposition
\begin{equation}\label{block dec 2021-11-10 11:00:13}
	kG=B_0\times \cdots \times B_l
\end{equation}
into the direct product of indecomposable \(k\)-algebras \(B_i\).
We call each indecomposable direct product component \(B_i\) a block of \(kG\) and the decomposition above the block decomposition. We remark that any block \(B_i\) is a two-sided ideal of \(kG\).

For any indecomposable \(kG\)-module \(U\), there exists a unique block \(B_i\) of \(kG\) such that \(U=B_iU\) and \(B_jU=0\) for all \(j\neq i\). Then we say that \(U\) lies in the block \(B_i\) or simply \(U\) is a \(B_i\)-module.
We denote by \(B_0(kG)\) the principal block of \(kG\), in which the trivial \(kG\)-module \(k_G\) lies.

Let \(G\) be a normal subgroup of a finite group \(\tilde{G}\), \(B\) a block of \(kG\) and \(\tilde{B}\) a block of \(k\tilde{G}\).
We say that \(\tilde{B}\) covers \(B\) (or that \(B\) is covered by \(\tilde{B}\)) if \(1_B 1_{\tilde{B}}\neq 0\).

\begin{remark}[{see \cite[Theorem 15.1, Lemma 15.3]{MR860771}}]\label{Remark:cover}
	With the notation above, the following are equivalent:
	\begin{enumerate}
		\item The block \(\tilde{B}\) covers \(B\).
		\item There exists a non-zero \(\tilde{B}\)-module \(U\) such that \(\restr_G^{\tilde{G}} U\) has a non-zero direct summand lying in \(B\).
		\item For any non-zero \(\tilde{B}\)-module \(U\), there exists a non-zero direct summand of \(\restr_G^{\tilde{G}} U\) lying in \(B\).
	\end{enumerate}
\end{remark}

We denote by \(\inertiagp_{\tilde{G}}(B)\) the  inertial group of \(B\) in \(\tilde{G}\), that is \(\inertiagp_{\tilde{G}}(B):=\left\{ x \in \tilde{G} \setmid xBx^{-1} = B \right\}\).

\begin{remark}
	The principal block \(B_0(kG)\) of \(kG\) is covered by the principal block \(B_0(k\tilde{G})\) of \(k\tilde{G}\) and \(\inertiagp_{\tilde{G}}(B_0(kG))=\tilde{G}\).
\end{remark}
\begin{remark}
	\label{remark another block 2022-08-30 16:09:09}
	Let \(G\) be a normal subgroup of a finite group \(\tilde{G}\), \(B\) a block of \(kG\) and \(M\) a \(B\)-module.
	Then \(xM\) is a \(B\)-module for \(x\in \tilde{G}\) if and only if \(x\in \inertiagp_{\tilde{G}}(B)\).
\end{remark}
\begin{proposition}[{see \cite[Theorem 5.5.10, Theorem 5.5.12]{MR998775}}]\label{Morita equivalence covering block}
	Let \(G\) be a normal subgroup of a finite group \(\tilde{G}\), \(B\) a block of \(kG\) and \(\beta\) a block of \(k\inertiagp_{\tilde{G}}(B)\) covering \(B\).
	Then the following hold:
	\begin{enumerate}
		\item For any \(B\)-module \(V\), the induced module \(\induc_G^{\inertiagp_{\tilde{G}}(B)}V\) is a direct sum of \(k\inertiagp_{\tilde{G}}(B)\)-module lying blocks covering \(B\).\label{Morita item direct sum 2022-06-09 01:49:36}
		\item There exists a unique block \(\tilde{B}\) of \(k\tilde{G}\) covering \(B\) such that the induction functor
		      \begin{equation}
			      \induc_{\inertiagp_{\tilde{G}}(B)}^{\tilde{G}} \colon k\inertiagp_{\tilde{G}}(B)\lmod \rightarrow k{\tilde{G}}\lmod
		      \end{equation}
		      restricts to a Morita equivalence
		      \begin{equation}\label{2020-03-25 15:17:09}
			      \begin{tikzcd}
				      \induc_{\inertiagp_{\tilde{G}}(B)}^{\tilde{G}} \colon \beta\lmod \ar[r]&  \tilde{B}\lmod.
			      \end{tikzcd}
		      \end{equation}
		      and the mapping \(\beta\) to \(\tilde{B}\) is a bijection between the set of blocks of \(k\inertiagp_{\tilde{G}(B)}\) covering \(B\) and the one of \(k\tilde{G}\) covering \(B\).\label{Morita item 2022-06-09 01:35:47}
	\end{enumerate}
\end{proposition}

\section{The main results and their applications}\label{section main 2022-07-13 09:12:20}
In this section, we give some lemmas and our main theorem.
After that, we give some applications and examples of the main results.
\subsection{Main theorems and their proof}
The next lemma plays a key role.
\begin{lemma}\label{commutative lemma 2022-06-02 15:41:42}
	Let \(G\) be a normal subgroup of a finite group \(\tilde{G}\) and \(M\) a \(kG\)-module satisfying \(xM \cong M\) as \(kG\)-modules for any \(x\in \tilde{G}\).
	Then the following hold:
	\begin{enumerate}
		\item \(xP(M) \cong P(M)\) for any \(x\in \tilde{G}\).\label{lem proj-covers 2022-06-02 12:35:13}
		\item \(x\Omega(M) \cong \Omega(M)\) for any \(x\in \tilde{G}\).\label{lem omega 2022-06-02 12:35:13}
		\item \(\induc_G^{\tilde{G}} \Omega(M)\cong \Omega(\induc_G^{\tilde{G}} M)\).\label{lem omega ind 2022-06-02 12:37:01}
		\item \(\tau(\induc_G^{\tilde{G}} M)\cong \induc_G^{\tilde{G}} \tau M\).\label{2020-03-25 01:17:54}
	\end{enumerate}
\end{lemma}
\begin{proof}
	For any \(x\in \tilde{G}\), we have an isomorphism \(\phi \colon xM\rightarrow M\) by the assumption.
	We consider the following commutative diagram in \(kG\lmod\) with exact rows:
	\begin{equation}
		\begin{tikzcd}
			0\ar[r]&x\Omega (M) \ar[r] \ar[d,"\phi''"]& xP (M) \ar[r,"{}^{x}\pi_M"]\ar[d,"\phi'"]& xM\ar[r] \ar[d,"\phi"] & 0\\
			0\ar[r]&\Omega(M) \ar[r] & P(M) \ar[r,"\pi_M"'] &  M\ar[r] &0.
		\end{tikzcd}
	\end{equation}
	Since \(\pi_M\) is an essential epimorphism and \(\phi\) is an isomorphism, the vertical morphisms \(\phi'\) and \(\phi''\) are isomorphisms and so \ref{lem proj-covers 2022-06-02 12:35:13} and \ref{lem omega 2022-06-02 12:35:13} holds.

	By \cref{Theorem: Frobenius and projective} \ref{proj to proj 2022-06-02 12:44:53}, we have the following commutative diagram in \(k\tilde{G}\lmod\) with exact rows:
	\begin{equation}
		\begin{tikzcd}[column sep=35pt]
			0\ar[r]&\induc_G^{\tilde{G}}\Omega(M) \ar[r] \ar[d,"\varphi'"]& \induc_G^{\tilde{G}}P(M) \ar[r,"\induc_G^{\tilde{G}}\pi_M"] \ar[d,"\varphi"]& \induc_G^{\tilde{G}} M\ar[r] \ar[d,"\Id_{\induc_G^{\tilde{G}} M}"]&0\\
			0\ar[r]&\Omega (\induc_G^{\tilde{G}} M) \ar[r]& P(\induc_G^{\tilde{G}} M) \ar[r,"\pi_{\induc_G^{\tilde{G}} M}"']&\induc_G^{\tilde{G}} M\ar[r]&0.
		\end{tikzcd}
	\end{equation}
	Since \(\pi_{\induc_G^{\tilde{G}} M}\) is an essential epimorphism, we have that the vertical morphisms \(\varphi\) and \(\varphi'\) are split epimorphisms, \(\Kernel \varphi \cong \Kernel \varphi'\) are projective \(k\tilde{G}\)-modules and that \(\Omega (\induc_G^{\tilde{G}} M) \oplus \Kernel \varphi'\cong \induc_G^{\tilde{G}}\Omega(M)\).
	By \cref{Mackey's decomposition formula 2022-08-28 15:16:37} and \ref{lem omega 2022-06-02 12:35:13}, we have
	\begin{align}
		\Omega(M)^{\oplus|\tilde{G}:G|}
		 & \cong\bigoplus_{x\in[\tilde{G}/G]}x\Omega(M)                                                            \\
		 & \cong \restr_G^{\tilde{G}}\induc_G^{\tilde{G}}\Omega(M)                                                 \\
		 & \cong \restr_G^{\tilde{G}}\Omega (\induc_G^{\tilde{G}} M) \oplus \restr_G^{\tilde{G}} \Kernel \varphi'.
	\end{align}
	Since \(\restr_G^{\tilde{G}} \Kernel \varphi'\) is projective by \cref{Theorem: Frobenius and projective} \ref{proj to proj 2022-06-02 12:44:53} and \(\Omega(M)\) has no non-zero projective summands by the self-injectivity of the group algebra \(kG\),
	we have that \(\Kernel \varphi\cong \Kernel \varphi' =0\).
	This finishes the proof of \ref{lem omega ind 2022-06-02 12:37:01}.

	Finally, we prove the assertion \ref{2020-03-25 01:17:54}.
	Since \(k\tilde{G}\) and \(kG\) are symmetric \(k\)-algebras, it holds that \(\tau M\cong \Omega \Omega(M)\) and \(\tau(\induc_G^{\tilde{G}} M)\cong \Omega \Omega (\induc_G^{\tilde{G}} M)\) for any \(kG\)-module \(M\).
	Therefore, \ref{2020-03-25 01:17:54} immediately follows from \ref{lem omega ind 2022-06-02 12:37:01}.
\end{proof}
\begin{theorem}\label{main group algebra 2022-06-09 01:38:53}
	Let \(G\) be a normal subgroup of a finite group \(\tilde{G}\) and \(M\) a support \(\tau\)-tilting \(kG\)-module satisfying \(xM\cong M\) as \(kG\)-modules for any \(x \in \tilde{G}\).
	Then the induced module \(\induc_G^{\tilde{G}}M\) of \(M\) is a support \(\tau\)-tilting \(k\tilde{G}\)-module.
\end{theorem}
\begin{proof}
	The similar proof of \cite[Theorem 4.2]{MR3531045} works in this setting.
	By \cref{commutative lemma 2022-06-02 15:41:42} \ref{lem omega ind 2022-06-02 12:37:01}, \cref{Theorem: Frobenius and projective} \ref{frob adjoint 2022-06-02 15:44:10}, \cref{Mackey's decomposition formula 2022-08-28 15:16:37}, the \(\inertiagp_{\tilde{G}}(B)\)-invariance of   and the \(\tau\)-rigidity of \(M\), we have the following:
	\begin{align}
		\Hom_{k\tilde{G}}(\induc_G^{\tilde{G}}M,\tau\induc_G^{\tilde{G}}M)
		 & \cong \Hom_{k\tilde{G}}(\induc_G^{\tilde{G}}M,\induc_G^{\tilde{G}}\tau M) \\
		 & \cong \Hom_{kG}(\restr_G^{\tilde{G}}\induc_G^{\tilde{G}}M,\tau M)         \\
		 & \cong \Hom_{kG}(\bigoplus_{x\in[\tilde{G}/G]}xM,\tau M)                   \\
		 & \cong \bigoplus_{x\in[\tilde{G}/G]} \Hom_{kG}(M,\tau M)                   \\
		 & = 0.
	\end{align}
	Therefore, we have that \(\induc_G^{\tilde{G}}M\) is \(\tau\)-rigid.
	By \cref{lem app 2022-06-02 16:07:34}, there exists an exact sequence
	\begin{equation}\label{ex app 2022-06-06 16:38:32}
		\begin{tikzcd}
			kG\ar[r,"f"]&M'\ar[r,"f'"]&M''\ar[r]&0
		\end{tikzcd}
	\end{equation}
	with \(M', M''\in \add M\) and \(f\) a left \(\add M\)-approximation of \(kG\).
	Applying the functor \(\induc_G^{\tilde{G}}\) to the exact sequence \eqref{ex app 2022-06-06 16:38:32}, we get the exact sequence
	\begin{equation}
		\begin{tikzcd}
			k\tilde{G}\cong \induc_G^{\tilde{G}}kG\ar[r,"\induc_G^{\tilde{G}}f"]&\induc_G^{\tilde{G}}M'\ar[r,"\induc_G^{\tilde{G}}f'"]&\induc_G^{\tilde{G}}M''\ar[r]&0
		\end{tikzcd}
	\end{equation}
	satisfying that \(\induc_G^{\tilde{G}}M',\induc_G^{\tilde{G}}M''\in\add \induc_G^{\tilde{G}}M\).
	Then by \cref{lem app 2022-06-02 16:07:34}, we only have to prove that \(\induc_G^{\tilde{G}}f\) is a left \(\add \induc_G^{\tilde{G}} M\)-approximation of \(k\tilde{G}\), that is, the map
	\begin{equation}\label{morphism any 2022-06-08 19:11:02}
		\begin{tikzcd}[column sep=35pt]
			\Hom_{k\tilde{G}}(\induc_G^{\tilde{G}}M',X)\ar[r,"\bullet \circ \induc_G^{\tilde{G}}f"]&\Hom_{k\tilde{G}}(k\tilde{G},X)
		\end{tikzcd}
	\end{equation}
	is surjective for any \(X\in \add \induc_G^{\tilde{G}}M\).
	First we prove that the map
	\begin{equation}\label{morphism 1 2022-06-08 19:11:36}
		\begin{tikzcd}[column sep=35pt]
			\Hom_{k\tilde{G}}(\induc_G^{\tilde{G}}M',\induc_G^{\tilde{G}}M)\ar[r,"\bullet \circ \induc_G^{\tilde{G}}f"]&\Hom_{k\tilde{G}}(k\tilde{G},\induc_G^{\tilde{G}}M)
		\end{tikzcd}
	\end{equation}
	is surjective.
	By \cref{Theorem: Frobenius and projective} \ref{frob adjoint 2022-06-02 15:44:10}, \cref{Mackey's decomposition formula 2022-08-28 15:16:37} and the assumption, we get the following commutative diagram:
	\begin{equation}
		\begin{tikzcd}[column sep=35pt]
			\Hom_{k\tilde{G}}(\induc_G^{\tilde{G}}M',\induc_G^{\tilde{G}}M)\ar[r,"\bullet \circ \induc_G^{\tilde{G}}f"]\ar[d,"\sim"',sloped]&\Hom_{k\tilde{G}}(k\tilde{G},\induc_G^{\tilde{G}}M)\ar[d,"\sim",sloped]\\
			\Hom_{kG}(M',\restr_G^{\tilde{G}}\induc_G^{\tilde{G}}M)\ar[r,"\bullet \circ f"]\ar[d,"\sim"',sloped]&\Hom_{kG}(kG,\restr_G^{\tilde{G}}\induc_G^{\tilde{G}}M)\ar[d,"\sim",sloped]\\
			\Hom_{kG}(M',\bigoplus_{x\in[\tilde{G}/G]}xM)\ar[r,"\bullet \circ f"]\ar[d,"\sim"',sloped]&\Hom_{kG}(kG,\bigoplus_{x\in[\tilde{G}/G]}xM)\ar[d,"\sim",sloped]\\
			\Hom_{kG}(M',M^{\oplus|\tilde{G}:G|})\ar[r,"\bullet \circ f"]&\Hom_{kG}(kG,M^{\oplus|\tilde{G}:G|}).
		\end{tikzcd}
	\end{equation}
	The map in the last row is surjective since \(f\) is left \(\add M\)-approximation of \(kG\),	which implies that the map in the first row, which is the map \eqref{morphism 1 2022-06-08 19:11:36}, is surjective.
	Hence, we get that
	\begin{equation}\label{morhism m direct 2022-06-08 19:10:37}
		\begin{tikzcd}[column sep=35pt]
			\Hom_{k\tilde{G}}(\induc_G^{\tilde{G}}M',\induc_G^{\tilde{G}}M^{\oplus m})\ar[r,"\bullet \circ \induc_G^{\tilde{G}}f"]&\Hom_{k\tilde{G}}(k\tilde{G},\induc_G^{\tilde{G}}M^{\oplus m})
		\end{tikzcd}
	\end{equation}
	is surjective for any \(m\in \mathbb{N}\).
	Now take \(X\in \add \induc_G^{\tilde{G}} M\) and \(h\in \Hom_{k\tilde{G}}(k\tilde{G},X)\) arbitrarily.
	Then there exists \(m\in \mathbb{N}\) and a split exact sequence
	\begin{equation}
		\begin{tikzcd}
			0\ar[r]&X\ar[r,"\alpha"]&\induc_G^{\tilde{G}}M^{\oplus m}\ar[r,"\beta"]&Y\ar[r]&0
		\end{tikzcd}
	\end{equation}
	in \(k\tilde{G}\lmod\). Let \(\gamma \colon \induc_G^{\tilde{G}}M^{\oplus m}\rightarrow X\) be a retraction of \(\alpha\), that is, a \(k\tilde{G}\)-homomorphism satisfying \(\gamma \circ \alpha=\Id_X\).
	Since the map \eqref{morhism m direct 2022-06-08 19:10:37} is surjective and \(\alpha\circ h \in \Hom_{k\tilde{G}}(k\tilde{G},\induc_G^{\tilde{G}}M^{\oplus m})\), there exists \(h'\in \Hom_{k\tilde{G}}(\induc_G^{\tilde{G}}M',\induc_G^{\tilde{G}}M^{\oplus m})\) such that \(h'\circ \induc_G^{\tilde{G}}f=\alpha\circ h\).
	Hence, we have that
	\begin{equation}
		h=\Id_X\circ h=\gamma \circ \alpha \circ h =\gamma \circ h'\circ \induc_G^{\tilde{G}}f.
	\end{equation}
	Therefore, the map \eqref{morphism any 2022-06-08 19:11:02} is surjective.
\end{proof}
The following result makes the assumption in \cref{main group algebra 2022-06-09 01:38:53} weaker in case where that the module \(M\) lies in \(B\) not only \(kG\)-module.
\begin{theorem}\label{main block 2022-07-13 08:57:19}
	Let \(G\) be a normal subgroup of a finite group \(\tilde{G}\), \(B\) a block of \(kG\), \(\tilde{B}\) a block of \(k\tilde{G}\) covering \(B\) and \(M\) a support \(\tau\)-tilting \(B\)-module satisfying \(xM\cong M\) as \(B\)-modules for any \(x \in \inertiagp_{\tilde{G}}(B)\).
	Then \(\induc_G^{\tilde{G}}M\) is a support \(\tau\)-tilting \(k\tilde{G}\)-module.
	In particular, \(\tilde{B}\induc_G^{\tilde{G}}M\) is a support \(\tau\)-tilting \(\tilde{B}\)-module.
\end{theorem}
\begin{proof}
	Let \(\tilde{B_1}=\tilde{B}, \ldots, \tilde{B_e}\) be the all blocks of \(k\tilde{G}\) covering \(B\).
	By \cref{Morita equivalence covering block} \ref{Morita item 2022-06-09 01:35:47}, we can take \(\beta_1, \ldots, \beta_e\) the blocks of \(k\inertiagp_{\tilde{G}}(B)\) satisfying the induction functor \(\induc_{\inertiagp_{\tilde{G}}(B)}^{\tilde{G}}\) restricts to a Morita equivalence

	\begin{equation}
		\begin{tikzcd}[column sep=20pt]
			\induc_{\inertiagp_{\tilde{G}}(B)}^{\tilde{G}} \colon \beta_i \lmod \ar[r]&\tilde{B_i}\lmod
		\end{tikzcd}
	\end{equation}
	for any \(i=1,\ldots, e\).
	By \cref{main group algebra 2022-06-09 01:38:53}, the induced module \(\induc_G^{\inertiagp_{\tilde{G}}(B)}M\) is a support \(\tau\)-tilting \(k\inertiagp_{\tilde{G}}(B)\)-module and hence \(\beta_i\induc_G^{\inertiagp_{\tilde{G}}(B)}M\) is a support \(\tau\)-tilting \(\beta_i\)-module for any \(i=1,\ldots, e\).
	Therefore, we have that \(\induc_{\inertiagp_{\tilde{G}}(B)}^{\tilde{G}}\beta_i\induc_G^{\inertiagp_{\tilde{G}}(B)}M\) is a support \(\tau\)-tilting \(\tilde{B_i}\)-module.
	By \cref{Morita equivalence covering block} \ref{Morita item direct sum 2022-06-09 01:49:36} and \cref{Theorem: Frobenius and projective} \ref{transitive induc 2022-06-09 02:03:23}, we have
	\begin{align}
		\bigoplus_{i=1}^e\induc_{\inertiagp_{\tilde{G}}(B)}^{\tilde{G}}\beta_i\induc_G^{\inertiagp_{\tilde{G}}(B)}M
		 & \cong \induc_{\inertiagp_{\tilde{G}}(B)}^{\tilde{G}}\bigoplus_{i=1}^e\beta_i\induc_G^{\inertiagp_{\tilde{G}}(B)}M \\
		 & \cong \induc_{\inertiagp_{\tilde{G}}(B)}^{\tilde{G}}\induc_G^{\inertiagp_{\tilde{G}}(B)}M                         \\
		 & \cong \induc_G^{\tilde{G}}M.
	\end{align}
	Hence, \(\induc_G^{\tilde{G}} M\) is a support \(\tau\)-tilting \(k\tilde{G}\)-module.
	Therefore, we get that \(\tilde{B}\induc_G^{\tilde{G}}M\) be a support \(\tau\)-tilting \(\tilde{B}\)-module.
\end{proof}
\begin{corollary}\label{ind ord 2022-08-28 14:58:52}
	Let \(G\) be a normal subgroup of a finite group \(\tilde{G}\), \(B\) a block of \(kG\) and \(\tilde{B}\) a block of \(k\tilde{G}\) covering \(B\).
	If \(M\geq M'\) in \(\stautilt B\) for \(\inertiagp_{\tilde{G}}(B)\)-invariant support \(\tau\)-tilting \(B\)-modules \(M\) and \(M'\), then \(\tilde{B}\induc_G^{\tilde{G}}M\geq \tilde{B}\induc_G^{\tilde{G}}M'\) in \(\stautilt \tilde{B}\).
\end{corollary}
\begin{proof}
	By the exactness of the induction functor \(\induc_G^{\tilde{G}}\) and \cref{main block 2022-07-13 08:57:19}, the statement is obvious.
\end{proof}
We will demonstrate that there is an interrelation between the orders of \(\inertiagp_{\tilde{G}}(B)\)-invariant support \(\tau\)-tilting \(B\)-modules and support \(\tau\)-tilting \(k\tilde{G}\)-modules.
\begin{proposition}\label{induc tau 2022-08-28 14:50:44}
	Let \(M\) be an \(\inertiagp_{\tilde{G}}(B)\)-invariant \(B\)-module.
	If the induced module \(\induc_G^{\tilde{G}}M\) is a support \(\tau\)-tilting \(k\tilde{G}\)-module, then \(M\) is a support \(\tau\)-tilting \(B\)-module.
\end{proposition}
\begin{proof}
	By \cref{tau number remark 2021-09-07 12:11:53}, we can take a projective \(k\tilde{G}\)-module \(\tilde{P}\) satisfying that \(\Hom_{k\tilde{G}}(\tilde{P},\induc_G^{\tilde{G}}M)=0\) and \(|\tilde{P}|+|\induc_G^{\tilde{G}}M|=|\tilde{B}|\).
	By \cref{tau tilting pair 2022-08-27 08:26:25} and \cref{Theorem: Frobenius and projective} \ref{proj to proj 2022-06-02 12:44:53}, we enough to show the following:
	\begin{enumerate}
		\item \(M\) is a \(\tau\)-rigid \(B\)-module.
		\item \(\Hom_B(B\restr_G^{\tilde{G}}\tilde{P},M)=0\).
		\item If \(\Hom_B(M,\tau X)=0\), \(\Hom_B(X, \tau M)=0\) and \(\Hom_B(B\restr_G^{\tilde{G}}\tilde{P},X)=0\), then \(X\in \add M\) for any \(B\)-module \(X\).
	\end{enumerate}
	By \cref{remark another block 2022-08-30 16:09:09}, the \(\inertiagp_{\tilde{G}}(B)\)-invariance of \(M\), \cref{Mackey's decomposition formula 2022-08-28 15:16:37}, \cref{Theorem: Frobenius and projective} \ref{frob adjoint 2022-06-02 15:44:10}, \cref{commutative lemma 2022-06-02 15:41:42} and the \(\tau\)-rigidity of \(\induc_G^{\tilde{G}}M\), we have that
	\begin{align}
		\Hom_B(M,\tau M)^{\oplus|\inertiagp_{\tilde{G}}(B):G|}
		 & \cong\Hom_{kG}(\bigoplus_{x\in [\inertiagp_{\tilde{G}}(B)/G]}xM,\tau M)
		\oplus \Hom_{kG}(\bigoplus_{\substack{x\in [\tilde{G}/G]
		\\x\notin \inertiagp_{\tilde{G}}(B)}}xM,\tau M)\\
		 & \cong\Hom_{kG}(\bigoplus_{x\in [\tilde{G}/G]}xM,\tau M)                  \\
		 & \cong\Hom_{kG}(\restr_G^{\tilde{G}}\induc_G^{\tilde{G}}M,\tau M)         \\
		 & \cong\Hom_{k\tilde{G}}(\induc_G^{\tilde{G}}M,\induc_G^{\tilde{G}}\tau M) \\
		 & \cong\Hom_{k\tilde{G}}(\induc_G^{\tilde{G}}M,\tau \induc_G^{\tilde{G}}M) \\
		 & = 0.
	\end{align}
	Hence, we have that the \(B\)-module \(M\) is a \(\tau\)-rigid \(B\)-module.
	Also, we have that
	\begin{align}
		\Hom_B(B\restr_G^{\tilde{G}}\tilde{P},M)
		 & \cong \Hom_{kG}(\restr_G^{\tilde{G}}\tilde{P},M)         \\
		 & \cong \Hom_{k\tilde{G}}(\tilde{P},\induc_G^{\tilde{G}}M) \\
		 & = 0.
	\end{align}
	For a \(B\)-module \(X\), we assume that \(\Hom_B(M,\tau X)=0\), \(\Hom_B(X,\tau  M)=0\) and \(\Hom_B(B\restr_G^{\tilde{G}}\tilde{P},X)=0\).
	Then by the \cref{commutative lemma 2022-06-02 15:41:42}, \cref{Theorem: Frobenius and projective} \ref{frob adjoint 2022-06-02 15:44:10}, \cref{Mackey's decomposition formula 2022-08-28 15:16:37}, \(\inertiagp_{\tilde{G}}(B)\)-invariance of \(M\) and the assumption above, we have that
	\begin{align}
		\Hom_{k\tilde{G}}(\induc_G^{\tilde{G}}M,\tau \induc_G^{\tilde{G}}X)
		 & \cong\Hom_{k\tilde{G}}(\induc_G^{\tilde{G}}M,\induc_G^{\tilde{G}}\tau X) \\
		 & \cong\Hom_{kG}(\restr_G^{\tilde{G}}\induc_G^{\tilde{G}}M,\tau X)         \\
		 & \cong\Hom_{kG}(\bigoplus_{x\in [\tilde{G}/G]}xM,\tau X)                  \\
		 & \cong\Hom_{kG}(\bigoplus_{x\in [\inertiagp_{\tilde{G}}(B)/G]}xM,\tau X)
		\oplus \Hom_{kG}(\bigoplus_{\substack{x\in [\tilde{G}/G]
		\\x\notin \inertiagp_{\tilde{G}}(B)}}xM,\tau X)\\
		 & \cong\Hom_B(M,\tau X)^{\oplus|\inertiagp_{\tilde{G}}(B):G|}              \\
		 & = 0.
	\end{align}
	Similarly, we have that \(\Hom_{k\tilde{G}}( \induc_G^{\tilde{G}}X,\tau\induc_G^{\tilde{G}}M)=0\).
	Also, we have that
	\begin{align}
		\Hom_{k\tilde{G}}(\tilde{P},\induc_G^{\tilde{G}}X)
		 & \cong \Hom_{kG}(\restr_G^{\tilde{G}}\tilde{P},X) \\
		 & \cong \Hom_{B}(B\restr_G^{\tilde{G}}\tilde{P},X) \\
		 & = 0.
	\end{align}
	Hence, we have that \(\induc_G^{\tilde{G}}X\in \add \induc_G^{\tilde{G}} M\) by \cref{tau tilting pair 2022-08-27 08:26:25}.
	Therefore, we have that \(\restr_G^{\tilde{G}}\induc_G^{\tilde{G}}X\in \add\restr_G^{\tilde{G}}\induc_G^{\tilde{G}}M\).
	In particular, we have that \(X\in \add M\) since \(\restr_G^{\tilde{G}}\induc_G^{\tilde{G}}X\cong \bigoplus_{x\in [\tilde{G}/G]}xX\) and \(\restr_G^{\tilde{G}}\induc_G^{\tilde{G}}M\cong \bigoplus_{x\in [\tilde{G}/G]}M\), which implies that \(M\) is a support \(\tau\)-tilting \(B\)-module by \cref{tau tilting pair 2022-08-27 08:26:25}
\end{proof}
\begin{theorem}\label{iff induc 2022-08-28 16:22:55}
	Let \(M\) and \(M'\) be \(\inertiagp_{\tilde{G}}(B)\)-invariant \(B\)-modules.
	Then the following hold:
	\begin{enumerate}
		\item \(M\) is a support \(\tau\)-tilting \(B\)-module if and only if \(\induc_G^{\tilde{G}}M\) is a support \(\tau\)-tilting \(k\tilde{G}\)-module.\label{iff stau 2022-08-28 14:48:21}
		\item Assume that \(M\) and \(M'\) are support \(\tau\)-tilting \(B\)-modules. Then \(M\geq M'\) in \(\stautilt B\) if and only if \(\induc_G^{\tilde{G}}M\geq \induc_G^{\tilde{G}}M'\) in \(\stautilt k\tilde{G}\).\label{iff ord 2022-08-28 14:49:24}
	\end{enumerate}
\end{theorem}
\begin{proof}
	\ref{iff stau 2022-08-28 14:48:21} is clear by \cref{main block 2022-07-13 08:57:19} and \cref{induc tau 2022-08-28 14:50:44}.
	In order to prove \ref{iff ord 2022-08-28 14:49:24}, we only show that if \(\induc_G^{\tilde{G}}M\geq \induc_G^{\tilde{G}}M'\) in \(\stautilt k\tilde{G}\) then \(M\geq M'\) in \(\stautilt B\) by \cref{ind ord 2022-08-28 14:58:52}, but it follows from the fact the restriction functor \(\restr_G^{\tilde{G}}\) is an exact functor, the \(\inertiagp_{\tilde{G}}(B)\)-invariance of \(M\) and \cref{Mackey's decomposition formula 2022-08-28 15:16:37}.
\end{proof}
\subsection{Examples}
Finally, we illustrate our main results with the following examples.
\begin{example}
	Let \(\tilde{G}\) be a finite group, \(G\) a normal subgroup of \(\tilde
	G\) with cyclic Sylow \(p\)-subgroup such that the quotient group \(\tilde{G}/G\) is a \(p\)-group, \(B\) a block of \(kG\) and \(\tilde{B}\) a block of \(k\tilde{G}\) covering \(B\).
	Then any support \(\tau\)-tilting \(B\)-module is \(\inertiagp_{\tilde{G}}(B)\)-invariant (see \cite[Lemma 4.2 (2)]{KK2}).
	Therefore, we can always apply \cref{main block 2022-07-13 08:57:19}.
\end{example}
\begin{example}
	Let \(G_1\) and \(G_2\) be arbitrary finite groups and \(M\) a support \(\tau\)-tilting \(kG_1\)-module.
	Then the group \(G_1\) is a normal subgroup of the direct product group \(G_1\times G_2\), and it is clear that \(M\cong xM\) for any \(x\in G_1\times G_2\).
	Therefore, the induced module \(\induc_{G_1}^{G_1\times G_2}M\cong kG_2\otimes_k M\) is support \(\tau\)-tilting \(k[G_1\times G_2]\)-module by \cref{main group algebra 2022-06-09 01:38:53}.
\end{example}
\begin{example}
	Let \(k\) be an algebraically closed field of characteristic \(p=2\), \(G\) the alternating group \(A_4\) of degree \(4\) and \(\tilde{G}\) the symmetric group \(S_4\) of degree \(4\).
	The principal blocks of \(kA_4\) and \(kS_4\) are themselves, respectively. Moreover, the block \(kA_4\) is covered by \(kS_4\).
	The algebras \(kA_4\) and \(kS_4\) are Brauer graph algebras associated to the Brauer graphs in \Cref{A4 Brauer graph 2022-07-16 22:29:08} and \Cref{S4 Brauer graph 2022-07-16 22:28:53}, respectively:

	\begin{figure}[ht]
		\begin{center}
			\subfigure[The Brauer graph of \(kA_4\)]{
				\begin{tikzpicture}
					\coordinate (O) at (0,0);
					\coordinate (A) at (4,0);
					\coordinate (B) at (2,3);
					\draw (O) edge node[below] {\(k_{A_4}=1\)} (A) ;
					\draw (A) edge node[right] {\(2\)} (B);
					\draw (B) edge node[left] {\(3\)}(O);
					\foreach \P in {O,A,B} \fill[white] (\P) circle (2pt);
					\foreach \P in {O,A,B} \draw (\P) circle (2pt);
				\end{tikzpicture}
				\label{A4 Brauer graph 2022-07-16 22:29:08}
			}
			\hfill
			\subfigure[The Brauer graph of \(kS_4\)]{

				\begin{tikzpicture}
					\node(m) at (5.5,1.5){multiplicity: \(2\)};
					\coordinate (O) at (2,1.5);
					\coordinate (A) at (4,1.5);
					\coordinate (B) at (1,1.5);
					\draw (O) edge node[below] {\(2'\)} (A) ;
					\foreach \P in {O,A} \draw (\P) circle (2pt);
					\draw[out=-90,in=-90] (A) edge node[below] {\(1'=k_{S_4}\)} (B);
					\draw[out=90,in=90] (A) edge node[below] {} (B);
					\fill[white] (A) circle (2pt);
					\fill[black] (O) circle (2pt);
				\end{tikzpicture}

				\label{S4 Brauer graph 2022-07-16 22:28:53}
			}
		\end{center}
		\vspace{-0.5cm}
		\caption{Brauer graphs}
		\label{Brauer 2022-07-16 12:36:42}
	\end{figure}
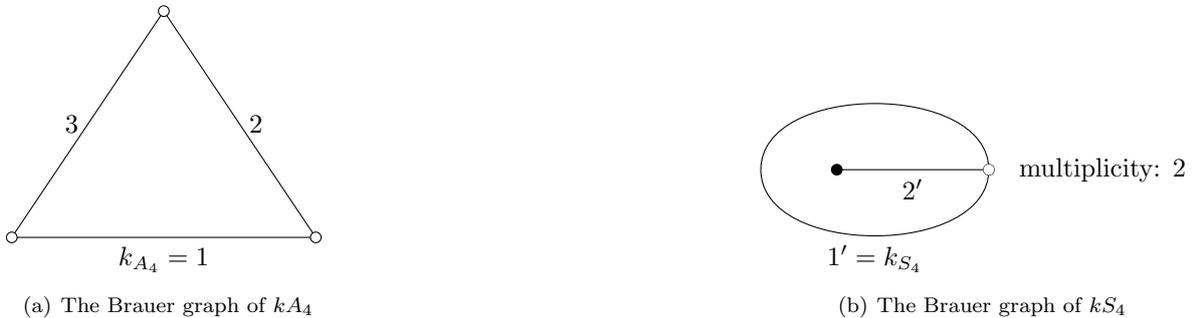

	Now we draw the Hasse diagram \(\mathcal{H}(\stautilt kA_4)\) of the partially ordered set \(\stautilt kA_4\) as follows:
	\begin{landscape}
		\begin{figure}[ht]
			\begin{tikzpicture}[scale=0.8]
				\node(hasse) at (-4,6) {\(\mathcal{H}(\stautilt kA_4):\)};
				\node(projene)[fill=orange] at (4,6){\(
					P(1)\oplus P(2)\oplus P(3)
					\)};
				\node(a) at (4,4){\(P_1\oplus
					\begin{smallmatrix}
						& 3 &   & 1 &   \\
						1 &   & 2 &   & 3 \\
					\end{smallmatrix}
					\oplus P_3\)};
				\node(j) at (4,0){\(
					\begin{smallmatrix}
						3\\
						1
					\end{smallmatrix}\oplus
					\begin{smallmatrix}
						& 2 &   & 3 &   \\
						3 &   & 3 &   & 2 \\
					\end{smallmatrix}\oplus
					\begin{smallmatrix}
						1\\
						3
					\end{smallmatrix}
					\)};
				\node(n) at (-0,0){\(
					\begin{smallmatrix}
						3\\
						1
					\end{smallmatrix}\oplus
					\begin{smallmatrix}
						3\\
						2
					\end{smallmatrix}\oplus
					P_3
					\)};
				\node(w) at (0,-2){\(
					\begin{smallmatrix}
						3\\
						1
					\end{smallmatrix}\oplus
					\begin{smallmatrix}
						3\\
						2
					\end{smallmatrix}\oplus
					3
					\)};
				\node(x) at (2,-4){\(
					\begin{smallmatrix}
						3\\
						1
					\end{smallmatrix}\oplus
					3\)};
				\node(z) at (-2,-4){\(
					\begin{smallmatrix}
						3\\
						2
					\end{smallmatrix}\oplus
					3
					\)};
				\node(y) at (0,-5){\(
					3
					\)};
				\node(k) at (4,-2){\(
					\begin{smallmatrix}
						3\\
						1
					\end{smallmatrix}\oplus
					\begin{smallmatrix}
						1\\
						3
					\end{smallmatrix}
					\)};
				\node(g) at (2,2){\(
					\begin{smallmatrix}
						3\\
						1
					\end{smallmatrix}\oplus
					\begin{smallmatrix}
						& 3 &   & 1 &   \\
						1 &   & 2 &   & 3 \\
					\end{smallmatrix}
					\oplus P_3
					\)};
				\node(h) at (6,2){\(
					P_1\oplus
					\begin{smallmatrix}
						& 3 &   & 1 &   \\
						1 &   & 2 &   & 3 \\
					\end{smallmatrix}\oplus
					\begin{smallmatrix}
						1\\
						3
					\end{smallmatrix}
					\)};
				\node(b)[fill=lime] at (-4,4){\(
					\begin{smallmatrix}
						& 2 &   & 3 &   \\
						3 &   & 1 &   & 2 \\
					\end{smallmatrix}
					\oplus P_2\oplus P_3
					\)};
				\node(c) at (12,4){\(
					P_1\oplus P_2\oplus
					\begin{smallmatrix}
						& 1 &   & 2 &   \\
						2 &   & 3 &   & 1 \\
					\end{smallmatrix}
					\)};
				\node(d) at (-6,2){\(
					\begin{smallmatrix}
						& 2 &   & 3 &   \\
						3 &   & 1 &   & 2 \\
					\end{smallmatrix}
					\oplus P_2\oplus
					\begin{smallmatrix}
						2\\3
					\end{smallmatrix}
					\)};
				\node(l)[fill=green] at (-4,0){\(
					\begin{smallmatrix}
						& 2 &   & 3 &   \\
						3 &   & 1 &   & 2 \\
					\end{smallmatrix}\oplus
					\begin{smallmatrix}
						3\\
						2
					\end{smallmatrix}\oplus
					\begin{smallmatrix}
						2\\
						3
					\end{smallmatrix}
					\)};
				\node(a1)[fill=olive] at (-4,-2){\(
					\begin{smallmatrix}
						3\\
						2
					\end{smallmatrix}\oplus
					\begin{smallmatrix}
						2\\
						3
					\end{smallmatrix}
					\)};
				\node(b1) at (-6,-4){\(
					2\oplus
					\begin{smallmatrix}
						2\\
						3
					\end{smallmatrix}
					\)};
				\node(c1) at (-8,-5){\(
					2
					\)};
				\node(m) at (-8,0){\(
					\begin{smallmatrix}
						2\\
						1
					\end{smallmatrix}\oplus
					P_2\oplus
					\begin{smallmatrix}
						2\\
						3
					\end{smallmatrix}
					\)};
				\node(pa1) at (-9,1){};
				\node(d1) at (-8,-2){\(
					\begin{smallmatrix}
						2\\
						1
					\end{smallmatrix}\oplus
					2\oplus
					\begin{smallmatrix}
						2\\
						3
					\end{smallmatrix}
					\)};
				\node(pa5) at (-9,-3){};
				\node(pa6) at (-9,-4){};
				\node(f) at (-2,2){\(
					\begin{smallmatrix}
						& 2 &   & 3 &   \\
						3 &   & 1 &   & 2 \\
					\end{smallmatrix}
					\oplus
					\begin{smallmatrix}
						3\\
						2
					\end{smallmatrix}
					\oplus
					P_3
					\)};
				\node(e) at (10,2){\(
					P_1\oplus
					\begin{smallmatrix}
						1\\
						2
					\end{smallmatrix}\oplus
					\begin{smallmatrix}
						& 1 &   & 2 &   \\
						2 &   & 3 &   & 1 \\
					\end{smallmatrix}
					\)};
				\node(u)[fill=cyan] at (8,0){\(
					P_1\oplus
					\begin{smallmatrix}
						1\\
						2
					\end{smallmatrix}\oplus
					\begin{smallmatrix}
						1\\
						3
					\end{smallmatrix}
					\)};
				\node(v)[fill=blue!20] at (8,-2){\(
					1\oplus
					\begin{smallmatrix}
						1\\
						2
					\end{smallmatrix}\oplus
					\begin{smallmatrix}
						1\\
						3
					\end{smallmatrix}
					\)};
				\node(o) at (12,0){\(
					\begin{smallmatrix}
						2\\
						1
					\end{smallmatrix}\oplus
					\begin{smallmatrix}
						1\\
						2
					\end{smallmatrix}\oplus
					\begin{smallmatrix}
						& 1 &   & 2 &   \\
						2 &   & 3 &   & 1 \\
					\end{smallmatrix}
					\)};
				\node(p) at (12,-2){\(
					\begin{smallmatrix}
						2\\
						1
					\end{smallmatrix}\oplus
					\begin{smallmatrix}
						1\\
						2
					\end{smallmatrix}
					\)};
				\node(t) at (14,-4){\(
					\begin{smallmatrix}
						2\\
						1
					\end{smallmatrix}\oplus
					2
					\)};
				\node(pa3) at (15,-3){};
				\node(pa4) at (15,-5){};
				\node(q) at (10,-4){\(
					1
					\oplus
					\begin{smallmatrix}
						1\\
						2
					\end{smallmatrix}
					\)};
				\node(r) at (6,-4){\(
					1
					\oplus
					\begin{smallmatrix}
						1\\
						3
					\end{smallmatrix}
					\)};
				\node(s)[fill=pink] at (8,-5){\(
					1
					\)};
				\node(i) at (14,2){\(
					\begin{smallmatrix}
						2\\
						1
					\end{smallmatrix}\oplus
					P_2\oplus
					\begin{smallmatrix}
						& 1 &   & 2 &   \\
						2 &   & 3 &   & 1 \\
					\end{smallmatrix}
					\)};
				\node(pa2) at (15,1){};
				\node(zero)[fill=red,opacity=0.6] at (0,-7){\(0\)};
				\draw[->](projene) edge node{} (a);
				\draw[->](projene) edge node{} (b);
				\draw[->](projene) edge node{} (c);
				\draw[->](b) edge node{} (d);
				\draw[->](a) edge node{} (g);
				\draw[->](a) edge node{} (h);
				\draw[->](b) edge node{} (f);
				\draw[->](d) edge node{} (l);
				\draw[->](d) edge node{} (m);
				\draw[->](f) edge node{} (l);
				\draw[->](c) edge node{} (e);
				\draw[->](c) edge node{} (i);
				\draw[->](y) edge node{} (zero);
				\draw[->](p) edge node{} (q);
				\draw[->](q) edge node{} (s);
				\draw[->](r) edge node{} (s);
				\draw[->](s) edge node{} (zero);
				\draw[->](c1) edge node{} (zero);
				\draw[->](b1) edge node{} (c1);
				\draw[->](i) edge node{} (o);
				\draw[->](e) edge node{} (o);
				\draw[->](e) edge node{} (u);
				\draw[->](o) edge node{} (p);
				\draw[->](x) edge node{} (y);
				\draw[->](z) edge node{} (y);
				\draw[->](g) edge node{} (n);
				\draw[->](f) edge node{} (n);
				\draw[->](h) edge node{} (j);
				\draw[->](h) edge node{} (u);
				\draw[->](u) edge node{} (v);
				\draw[->](v) edge node{} (q);
				\draw[->](v) edge node{} (r);
				\draw[->](k) edge node{} (r);
				\draw[->](k) edge node{} (x);
				\draw[->](g) edge node{} (j);
				\draw[->](m) edge node{} (d1);
				\draw[->](l) edge node{} (a1);
				\draw[->](w) edge node{} (x);
				\draw[->](w) edge node{} (z);
				\draw[->](a1) edge node{} (z);
				\draw[->](a1) edge node{} (b1);
				\draw[->](d1) edge node{} (b1);
				\draw[->](p) edge node{} (t);
				\draw[->](n) edge node{} (w);
				\draw[->](j) edge node{} (k);
				\draw[->](pa1) edge node{} (m);
				\draw[-](i) edge node{} (pa2);
				\draw[->](pa3) edge node{} (t);
				\draw[-](t) edge node{} (pa4);
				\draw[-](d1) edge node{} (pa5);
				\draw[->](pa6) edge node{} (c1);
			\end{tikzpicture}
			\caption[]{The Hasse diagram of \(\stautilt kA_4\)}
			\label{Hasse kA4 2022-07-07 02:43:05}
		\end{figure}
	\end{landscape}
	\noindent
	The colored support \(\tau\)-tilting modules in \Cref{Hasse kA4 2022-07-07 02:43:05} are all of the invariant support \(\tau\)-tilting modules under the action of \(S_4\).
	Next we draw the Hasse diagram \(\mathcal{H}(\stautilt kS_4)\) of partially ordered set \(\stautilt kS_4\) as follows:
	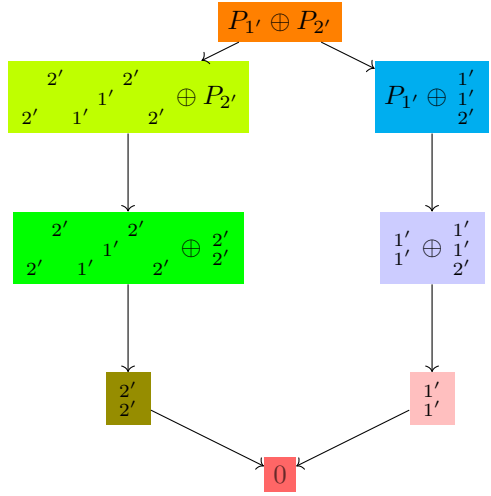
\begin{figure}[ht]
		\centering
		\begin{tikzpicture}
			\node(H) at (-5,1){\(\mathcal{H}(\stautilt(kS_4)):\)};
			\node(a)[fill=orange] at (-2,0){\(
				P_{1'}\oplus P_{2'}
				\)};
			\node(b)[fill=lime] at (-4,-1){\(
				\begin{smallmatrix}
					&2'&&&2'&\\
					&&&1'&&\\
					2'&&1'&&&2'
				\end{smallmatrix}
				\oplus P_{2'}
				\)};
			\node(d)[fill=green] at (-4,-3){\(
				\begin{smallmatrix}
					&2'&&&2'&\\
					&&&1'&&\\
					2'&&1'&&&2'
				\end{smallmatrix}
				\oplus
				\begin{smallmatrix}
					2'\\
					2'
				\end{smallmatrix}
				\)};
			\node(e)[fill=olive] at (-4,-5){\(
				\begin{smallmatrix}
					2'\\
					2'
				\end{smallmatrix}
				\)};
			\node(zero)[fill=red,opacity=0.6] at (-2,-6){\(
				0
				\)};
			\node(c)[fill=cyan] at (0,-1){\(
				P_{1'}\oplus
				\begin{smallmatrix}
					1'\\
					1'\\
					2'
				\end{smallmatrix}
				\)};
			\node(f)[fill=blue!20] at (0,-3){\(
				\begin{smallmatrix}
					1'\\
					1'
				\end{smallmatrix}
				\oplus
				\begin{smallmatrix}
					1'\\
					1'\\
					2'
				\end{smallmatrix}
				\)};
			\node(g)[fill=pink] at (0,-5){\(
				\begin{smallmatrix}
					1'\\
					1'
				\end{smallmatrix}
				\)};
			\draw[->](a) edge node{} (b);
			\draw[->](a) edge node{} (c);
			\draw[->](b) edge node{} (d);
			\draw[->](d) edge node{} (e);
			\draw[->](e) edge node{} (zero);
			\draw[->](g) edge node{} (zero);
			\draw[->](f) edge node{} (g);
			\draw[->](c) edge node{} (f);
		\end{tikzpicture}
		\caption{The Hasse diagram of \(\stautilt kS_4\)}
		\label{Hasse kS4 2022-07-07 02:39:33}
	\end{figure}

	\noindent
	The functor \(\induc_{A_4}^{S_4}\) takes each colored \(S_4\)-invariant support \(\tau\)-tilting \(kA_4\)-module in \Cref{Hasse kA4 2022-07-07 02:43:05} to that in \Cref{Hasse kS4 2022-07-07 02:39:33} with the same color.
	We remark that even if support \(\tau\)-tilting \(kA_4\)-module \(M\) is basic, its induction \(\induc_{A_4}^{S_4}M\) is not necessarily basic.
	For example, the induced module \(\induc_{A_4}^{S_4}(
	1\oplus
	\begin{smallmatrix}
		1\\
		2
	\end{smallmatrix}\oplus
	\begin{smallmatrix}
		1\\
		3
	\end{smallmatrix})
	\cong
	\begin{smallmatrix}
		1'\\
		1'
	\end{smallmatrix}
	\oplus
	\begin{smallmatrix}
		1'\\
		1'\\
		2'
	\end{smallmatrix}
	\oplus
	\begin{smallmatrix}
		1'\\
		1'\\
		2'
	\end{smallmatrix}
	\)
	is not basic.
\end{example}

\subsection*{Acknowledgements}
The author would like to thank Yuta~Kozakai for useful advice, discussions and comments.


\Addresses
\end{document}